\documentclass[12pt]{amsart}
\usepackage{amssymb}
\usepackage[margin=1in]{geometry}

\DeclareMathOperator{\supp}{supp}

\DeclareMathOperator{\Tr}{Tr}

\DeclareMathOperator{\meas}{meas}
\DeclareMathOperator{\Op}{Op}
\DeclareMathOperator{\AD}{AD}
\DeclareMathOperator*{\slim}{s-lim}

\renewcommand{\Im}{\hbox{{\rm Im}}\,}

\newcommand{\abs}[1]{\lvert#1\rvert}
\newcommand{\norm}[1]{\lVert#1\rVert}
\newcommand{\Norm}[1]{\left\lVert#1\right\rVert}
\newcommand{\jap}[1]{\langle#1\rangle}

\newcommand{\N}{{\mathbb N}}
\newcommand{\R}{{\mathbb R}}

\newcommand{\dd}{{\mathrm{d}}}
\newcommand{\p}{{\rho}}

\newcommand{\w}{{\omega}}
\newcommand{\e}{{\varepsilon}}
\newcommand{\x}{{\langle x \rangle}}
\newcommand{\sph}{{{\mathbb S}^{d-1}}}

\numberwithin{equation}{section}

\newtheorem{theorem}{Theorem}[section]
\newtheorem{lemma}[theorem]{Lemma}

\newtheorem{proposition}[theorem]{Proposition}

\theoremstyle{definition}

\theoremstyle{remark}

\newtheorem*{remark*}{\bf Remark}

\numberwithin{equation}{section}



\begin{document}

\title[Spectral density of the scattering matrix]{The spectral density of the scattering matrix for  high energies}

\author{Daniel Bulger and Alexander Pushnitski}
\address{Department of Mathematics,
King's College London,
Strand, London, WC2R~2LS, U.K.}
\email{daniel.bulger@kcl.ac.uk}
\email{alexander.pushnitski@kcl.ac.uk}

\begin{abstract}
We determine the 
density of eigenvalues of the scattering matrix 
of the Schr\"odinger operator with  a short range potential
in the high energy asymptotic regime. 
We give an explicit formula for this density 
in terms of the X-ray transform of the potential. 
\end{abstract}

\maketitle

\section{Introduction and Main Result}
\label{sec.a}
\subsection{Introduction}
\label{sec.a0}
The object of study in this paper is the (on-shell) scattering matrix $S(\lambda)$ 
corresponding to the scattering of a $d$-dimensional quantum particle on an 
external short range potential $V$ at the energy $\lambda>0$. 
The scattering matrix $S(\lambda)$ 
is a unitary operator in $L^2(\mathbb S^{d-1})$ and the difference $S(\lambda)-I$ 
is compact. Thus, the eigenvalues of $S(\lambda)$ can be written as 
\begin{equation}
\label{eq.a1a}
\exp(i\theta_{n}(\lambda)),\quad
\theta_{n}(\lambda)\in [-\pi,\pi), 
\quad n\in\N
\end{equation}
and $\theta_n(\lambda)\to0$ as $n\to\infty$. 
The quantities $\theta_n(\lambda)$ are known as scattering phases. 
The scattering phases are usually discussed in physics literature
(see e.g. \cite[Section~123]{LL}) under the additional assumption of the spherical symmetry 
of $V$; here we do not need this assumption. 

Our aim is to study the asymptotic distribution of scattering phases
$\{\theta_n(\lambda)\}_{n=1}^\infty$ when $\lambda\to\infty$. 
It turns out that after an appropriate scaling the asymptotic density of scattering phases
can be described by a simple explicit formula involving the X-ray transform 
(see \eqref{eq.a5}) of the potential $V$. 
This formula has a semiclassical nature. 

The key idea of this paper goes back to the work of M.~Sh.~Birman
and D.~R.~Yafaev \cite{BYa1} (see also \cite{BI1}), 
where the asymptotics of $\theta_n(\lambda)$ 
for a \emph{fixed} $\lambda$ and $n\to\infty$ was determined for some 
class of potentials $V$. 
This asymptotic behaviour is \emph{not} uniform in $\lambda$, and thus our 
results cannot be derived from those of \cite{BYa1}.  
However, both the results of \cite{BYa1} and our results are based on the following 
observation. The leading term of the asymptotics of $\theta_n(\lambda)$ 
(in both asymptotic regimes) is determined by the Born approximation of 
the scattering matrix. The Born approximation is essentially a pseudodifferential 
operator ($\Psi$DO) on the sphere $\mathbb S^{d-1}$ with the symbol given by 
the X-ray transform of $V$. 
Standard $\Psi$DO results can be used to give spectral asymptotics for 
operators of such type. In both asymptotic regimes, the desired spectral asymptotics
are given by a Weyl type formula involving the symbol of the $\Psi$DO.

\subsection{The scattering matrix}
\label{sec.a1}
Let us briefly recall the relevant definitions. 
Let $H_{0}=-\Delta$ and $H=-\Delta+V$ be the Schr\"odinger operators
in $L^{2}(\R^{d})$, $d\geq 2$, 
where $V$ is the operator of multiplication by a real-valued potential $V\in C(\R^{d})$ 
which is assumed to  satisfy the estimate
\begin{equation}
\label{eq.a1}
\abs{V(x)}\leq\frac{C}{(1+\abs{x})^{\p}}, \quad \p>1
\end{equation}
with some constant $C>0$.  
It is one of the fundamental facts of scattering theory 
\cite{Kato,Kuroda}
that under the assumption \eqref{eq.a1} the wave operators
\begin{equation*}
W_{\pm}=\slim_{t\rightarrow\pm\infty}e^{itH}e^{-itH_{0}}
\end{equation*}
exist and are complete;
the scattering operator $\mathbf S=W_+^*W_-$ is unitary in $L^{2}(\R^{d})$ and commutes with $H_0$. 
Let 
$F: L^2(\R^d)\to L^2((0,\infty);L^2(\sph))$ be the unitary operator
$$
(Fu)(\lambda,\omega)=\frac1{\sqrt{2}}\lambda^{(d-2)/4}\widehat u(\sqrt{\lambda}\omega), 
\quad \lambda>0,\quad \omega\in\sph,
$$
where $\widehat u$ is the usual (unitary) Fourier transform of $u$. 
The operator $F$ diagonalises $H_0$, i.e.
$$
(FH_0 u)(\lambda,\omega)
=
\lambda (Fu)(\lambda,\omega),
\quad
\forall u\in C_0^\infty(\R^{d}).
$$
The commutation relation $\mathbf SH_0=H_0\mathbf S$ 
implies that $F$ also diagonalises $\mathbf S$, i.e.
$$
(F\mathbf Su)(\lambda,\cdot)=S(\lambda)(Fu)(\lambda,\cdot),
$$
where $S(\lambda):L^2(\sph)\to L^2(\sph)$ 
is the unitary operator known as the (on-shell) scattering matrix.
See e.g. the book \cite{Ya2} for the details. 

Under the assumption 
\eqref{eq.a1} the operator $S(\lambda)-I$ is compact and consequently the eigenvalues
of  $S(\lambda)$ (enumerated with multiplicities taken into account)
can be written as \eqref{eq.a1a} with $\theta_n(\lambda)\to0$ 
as $n\to\infty$.

\subsection{The purpose of the paper}
\label{sec.a2}
Our purpose is to describe the asymptotic density of the eigenvalues of 
$S(\lambda)$ as $\lambda\rightarrow\infty$. 
We recall the estimate (see e.g. \cite[Section~8.1]{Ya2})
\begin{equation}
\Vert S(\lambda) - I\Vert = O(\lambda^{-1/2}),\quad \lambda\rightarrow\infty.
\label{eq.a2}
\end{equation}
This estimate is sharp, i.e. $O(\lambda^{-1/2})$ cannot be replaced by 
$o(\lambda^{-1/2})$; this can be seen by considering the case of a spherically symmetric potential 
and using the separation of variables.
Thus, the spectrum of the scattering matrix $S(\lambda)$ 
for large $\lambda$ 
consists of a cluster of eigenvalues located on an arc of length $O(\lambda^{-1/2})$ 
around $1$. Let us define the eigenvalue counting measure for $S(\lambda)$. 
The estimate \eqref{eq.a2} suggests the following scaling: for $\lambda\geq1$, 
set $k=\sqrt{\lambda}>0$ 
and define (using notation \eqref{eq.a1a})
\begin{equation}
\label{eq.a3}
\mu_{k}([\alpha,\beta])
=
\#\{ n\in\N : \alpha\leq k\theta_{n}(k^2)\leq\beta\},
\quad
[\alpha,\beta]\subset\R\backslash\{ 0\}
\end{equation}
where $\#$ denotes the number of elements in the set. 
We will study the weak asymptotics of $\mu_{k}$ as $k\rightarrow\infty$,
i.e. we consider the asymptotics of the integrals 
$$
\int_{-\infty}^{\infty}\psi(t)\dd\mu_{k}(t),\quad k\rightarrow\infty
$$
for test functions $\psi\in C_{0}^{\infty}(\R\backslash\{ 0\})$.

\subsection{Main result}
\label{sec.a3}
In order to describe the weak limit of the measures $\mu_k$, we need to fix some notation.
For any $\w\in\sph$, let $\Lambda_\omega\subset\R^{d}$ denote the hyperplane passing through the origin and orthogonal to $\w$. 
We equip both $\sph$ and $\Lambda_\omega$ with the standard $(d-1)$-dimensional Lebesgue measure
(=Euclidean area).
We set
\begin{equation}
\label{eq.a5}
X(\w,\eta)=-\frac{1}{2}\int_{-\infty}^{\infty}V(t\w +\eta)\dd t,
\quad
\omega\in\sph,
\quad 
\eta\in\Lambda_\omega.
\end{equation}
The function $X$ (up to a multiplicative factor) is known as the X-ray transform of $V$ in the inverse problem literature.
The following elementary estimate is a direct consequence of \eqref{eq.a1}:
\begin{equation}
\label{eq.a6}
\abs{X(\w,\eta)}\leq C(V)(1+\abs{\eta})^{1-\p},
\quad \omega\in\sph,
\quad \eta\in\Lambda_\omega
\end{equation}
with some constant $C(V)$.
We define a measure $\mu$ on $\R\backslash\{ 0\}$ by
$$
\mu([\alpha,\beta]) 
= 
(2\pi)^{1-d}\meas
\lbrace(\w,\eta)\in\sph\times\Lambda_\omega : 
\alpha\leq X(\w,\eta)\leq\beta\rbrace,\quad [\alpha,\beta]\subset\R\backslash\{ 0\},
$$
where $\meas$ denotes the usual product measure. 
By the boundedness of $V$, the measure $\mu$ has a compact support. 
The measure $\mu$ need not be absolutely continuous. 
The measure $\mu$ may be weakly singular at zero in the following sense: 
$\mu((0,\infty))$ or $\mu((-\infty,0))$ may be infinite, but, 
by the estimate \eqref{eq.a6} we have 
\begin{equation}
\int_{-\infty}^{\infty}\abs{t}^{\ell}\dd\mu(t)<\infty,
\quad \forall 
\ell>(d-1)/(\rho-1).
\label{a6a}
\end{equation}

Our main result is as follows:
\begin{theorem}
\label{thm.a1}
Let $V\in C(\R^{d})$ be a potential satisfying \eqref{eq.a1}. 
Then for any test function $\psi\in C_0^{\infty}(\R\backslash\{ 0\})$,
\begin{equation}
\label{eq.a8}
\lim_{k\rightarrow\infty}k^{1-d}\int_{-\infty}^{\infty}\psi(t)\dd\mu_{k}(t) = \int_{-\infty}^{\infty}\psi(t)\dd\mu(t).
\end{equation}
\end{theorem}
This can be more succinctly put as the weak convergence of the measures
\begin{equation}
k^{1-d}\mu_k\to\mu, \quad k\to\infty.
\label{eq.a8a}
\end{equation}

Much of the inspiration for both the content of this paper and the proofs 
may be found in \cite{AP1}, where similar asymptotics are determined for the 
spectrum of the Landau Hamiltonian (i.e.\ two-dimensional Schr\"odinger operator 
with a constant homogeneous magnetic field)
perturbed by a potential which obeys the same condition \eqref{eq.a1}.

We would like to mention also the paper \cite{ZZ} where
the high energy asymptotic distribution of the phases $\theta_n(\lambda)$ was studied
for scattering problems on manifolds of  a certain special class. 
The results of \cite{ZZ} are much more detailed than ours and include the 
asymptotics of the pair correlation measure.

\subsection{Comparison with \cite{BYa1}}
In \cite{BYa1}, the case of potentials with the power asymptotics at infinity
of the type
\begin{equation}
V(x)
=
v(x/\abs{x})\abs{x}^{-\rho}(1+o(1)), 
\quad
\abs{x}\to\infty, 
\quad 
\rho>1,
\label{eq.a14a}
\end{equation}
was considered. 
Using our notation $\mu_k$, $\mu$, the result of \cite{BYa1}
can be written as
\begin{equation}
\begin{split}
k^{1-d}\mu_k((\e,\infty))&\sim
\mu((\e,\infty)),
\\
k^{1-d}\mu_k((-\infty,-\e))&\sim
 \mu((-\infty,-\e)),
\end{split}
\label{eq.a14c}
\end{equation}
when $k>0$ is fixed and $\e\to+0$. Here $a\sim b$ means $\frac{a}{b}\to1$.

Clearly, our main result \eqref{eq.a8a} is expressed by the same formula
as \eqref{eq.a14c}, but the asymptotic regimes are different. Neither of the
results \eqref{eq.a8a}, \eqref{eq.a14c} implies the other one.

\subsection{Semiclassical interpretation}
\label{sec.a5}
By the definition of the scattering operator $\mathbf{S}$, for any $\psi\in L^{2}(\R^{d})$ we have
\begin{equation*}
\begin{split}
i((\mathbf{S}-I)\psi,\psi) 
&= i\lim_{t\rightarrow\infty} ((e^{-2itH}e^{itH_{0}}\psi,e^{-itH_{0}}\psi)-\Vert\psi\Vert^{2}) 
= i\int_{0}^{\infty}\frac{\dd}{\dd t}(e^{-2itH}e^{itH_{0}}\psi,e^{-itH_{0}}\psi)\dd t\\
&=\int_{0}^{\infty} (Ve^{-2itH}e^{itH_{0}}\psi,e^{-itH_{0}}\psi)\dd t
+ \int_{0}^{\infty} (Ve^{itH_{0}}\psi,e^{2itH}e^{-itH_{0}}\psi)\dd t.
\end{split}
\end{equation*}
If $\psi$ corresponds to large energies, the right hand side can be approximated by the first term 
in its expansion in powers of $V$. This means that we can replace $e^{itH}$ by $e^{itH_0}$ 
in the above expressions, and so  
\begin{equation}
\label{eq.a15}
i((\mathbf{S}-I)\psi,\psi) \approx \int_{-\infty}^{\infty} (Ve^{-itH_{0}}\psi,e^{-itH_{0}}\psi)\dd t, \quad\psi\in L^{2}(\R^{d}),
\end{equation}
which is exactly the Born approximation in the time-dependent picture. 

In order to write down the classical analogue of the right hand side of \eqref{eq.a15}, 
assume that $\psi$ is concentrated near $x$ in the coordinate representation and near $p$ in the 
momentum representation. Then $\psi$ represents a particle with the coordinate $x$ and momentum $p$, 
and in the same way $e^{-itH_{0}}\psi$ represents a particle with the coordinate $x+2pt$ and momentum $p$. 
Thus, the classical analogue of the right hand side of \eqref{eq.a15} is 
\begin{equation*}
\int_{-\infty}^{\infty}V(x+2pt)\dd t=\frac{1}{2\vert p\vert}\int_{-\infty}^{\infty}V(x+\w t')\dd t',\quad (x,p)\in\R^{d}\times\R^{d},
\end{equation*}
where $\w=\frac{p}{\vert p\vert}\in\sph$. 
This calculation 
explains the appearance of the X-ray transform  in the asymptotics of $S(\lambda)$.

\subsection{Key steps of the proof}
\label{sec.a4}
First we recall the stationary representation for the scattering matrix.
For $k>0$ and $\rho>1$, we define the operator
$\Gamma_\rho(k):L^2(\R^d)\to L^2(\sph)$ by 
\begin{equation}
\label{eq.a9}
(\Gamma_\rho(k)u)(\omega)
=
\frac1{\sqrt{2}}
k^{(d-2)/2}(2\pi)^{-d/2}
\int_{\R^d} u(x)\langle x\rangle^{-\rho/2} e^{-ik\jap{x,\omega}}
\dd x,\quad\w\in\sph,
\end{equation}
where $\langle x\rangle=(1+\abs{x}^2)^{1/2}$. 
By the Sobolev trace theorem, $\Gamma_\rho(k)$ is bounded for $\rho>1$. 
Next, let 
$$
T(z)=\langle x\rangle^{-\rho/2}(H-zI)^{-1}\langle x\rangle^{-\rho/2}, 
\quad 
\Im z>0;
$$
according to the limiting absorption principle, the limits $T(k^2\pm i0)$ 
exist in the operator norm for all $k>0$. 
Denote by $J$ the bounded operator of multiplication by $\langle x\rangle^{\rho}V(x)$ in $L^2(\R^d)$. 
The stationary representation for the scattering matrix can be written as (see e.g.\cite[Section~6.6]{Ya2})
$$
S(k^2)=I-2\pi i \Gamma_\rho(k)(J-JT(k^2+i0)J)\Gamma_\rho(k)^*, \quad k>0.
$$
The asymptotic density of eigenvalues of $S(k^2)$ for large $k$ is determined 
by the \emph{Born approximation} of the scattering matrix, defined as
\begin{equation}
\label{eq.a12}
S_B(k^2)=I-2\pi i \Gamma_\rho(k)J\Gamma_\rho(k)^*, \quad k>0.
\end{equation}
The key observation due to M.~Birman and D.~Yafaev \cite{BYa1}
is that the operator $\Gamma_\rho(k)J\Gamma_\rho(k)^*$ in $L^{2}(\sph)$
with the integral kernel
$$
2^{-1}k^{d-2}(2\pi)^{-d}\int_{\R^{d}}e^{-ik\jap{\w-\w',x}}V(x)\dd x,
\quad\w,\w'\in\sph
$$
can be represented as a $\Psi$DO
on the sphere with the symbol (up to inessential constants) $X(\omega,\eta)$.  
We combine this observation with the  standard semiclassical 
pseudodifferential techniques to 
obtain the spectral asymptotics of $S_B(k^2)$. 
In this way we prove the asymptotic formula (see Lemma~\ref{prop.c2})
\begin{equation}
\label{eq.a14}
\lim_{k\rightarrow\infty}k^{1-d}\Tr(\Im kS_{B}(k^{2}))^{\ell} 
= 
\int_{-\infty}^{\infty}t^{\ell}\dd\mu(t)
\end{equation}
for any natural number $\ell>(d-1)/(\rho-1)$; 
note that the r.h.s. of \eqref{eq.a14} is finite by \eqref{a6a}.

Using \eqref{eq.a14} and the estimates for the Schatten norm of $S(k^2)-S_B(k^2)$ 
we prove that \eqref{eq.a8} holds true for test functions $\psi(t)$ which coincide with $t^{\ell}$ 
for all sufficiently small $t$.
Theorem~\ref{thm.a1} follows by an application of the Weierstrass approximation theorem.

\subsection{Acknowledgements}
\label{sec.a6}
The authors are grateful to  Y.~Safarov
for expert advice concerning the $\Psi$DO aspect of this work.
The authors are grateful to 
N.~Filonov and D.~Yafaev for a number of useful discussions 
and remarks on the text of the paper.

\section{Preliminary statements}
\label{sec.b}

\subsection{The limiting absorption principle and its consequences}
\label{sec.b1}
First we need some notation. 
We denote by $S_{p}$, $p\geq1$,  the usual Schatten class and by $\Vert\cdot\Vert_{p}$ 
the associated norm. 
Let $X_{\rho}$ be the normed linear space of all potentials $V\in C(\R^{d})$ satisfying \eqref{eq.a1} with the norm
$$
\Norm{V}_{X_{\rho}} = \sup_{x\in\R^{d}}\abs{V(x)}\x^{\rho}.
$$
We recall the following estimates from \cite{Ya2}:
\begin{proposition}\label{prop.b1}
Let $V\in X_\rho$ with some $\rho>1$. 
Then for any $\ell\geq 1$ satisfying $\ell>\frac{d-1}{\p-1}$, one has
\begin{align}
\label{eq.b3}
&\sup_{k\geq 1}k^{\frac{1-d}{\ell}}\Vert k\Im S_{B}(k^2)\Vert_{\ell}
\leq 
C(\ell,\p,d)\norm{V}_{X_\rho},
\\
\label{eq.b3b}
&\sup_{k\geq 1}k^{\frac{1-d}{\ell}}\Vert k^2 \Im (S_{B}(k^2)-S(k^2))\Vert_{\ell}
\leq 
C(\ell,\p,d,V),
\\
\label{eq.b3a}
&\sup_{k\geq 1}k^{\frac{1-d}{\ell}}\Vert k\Im S(k^2)\Vert_{\ell}
\leq 
C(\ell,\p,d,V).
\end{align}
\end{proposition}
The estimate \eqref{eq.b3} is a direct consequence of \cite[Proposition~8.1.3]{Ya2}.
The estimate \eqref{eq.b3b} is proven in \cite[Proposition~8.1.4]{Ya2}.
The estimate \eqref{eq.b3a} is a direct consequence of \eqref{eq.b3} and \eqref{eq.b3b}.

\begin{lemma}\label{prop.b2}
Let $V\in X_\rho$ with $\rho>1$. 
Then for any integer $\ell\geq 1$ satisfying $\ell>\frac{d-1}{\p-1}$, 
one has
$$
\vert\Tr(k\Im S(k^2))^{\ell} - \Tr(k\Im S_{B}(k^2))^{\ell}\vert = O(k^{d-2}), 
\quad 
k\to\infty.
$$
\end{lemma}
\begin{proof}
From
\begin{equation*}
A^{\ell} - B^{\ell} = \sum_{j=0}^{\ell-1}A^{j}(A-B)B^{\ell-1-j}
\end{equation*}
one easily obtains
\begin{equation}
\label{eq.b10}
\vert\Tr(A^{\ell} - B^{\ell})\vert
\leq
\ell \Vert A-B\Vert_{\ell}\max\{\Vert A\Vert_{\ell}^{\ell-1}, \Vert B\Vert_{\ell}^{\ell-1}\},\quad A,B\in S_{\ell}.
\end{equation}
Thus, it suffices to prove the relation 
$$
\norm{k\Im(S(k^2)-S_B(k^2))}_{\ell}
\max\{\Vert k\Im S(k^2)\Vert_{\ell}^{\ell-1}, \Vert k\Im S_{B}(k^2)\Vert_{\ell}^{\ell-1}\}
=
O(k^{d-2}), 
\quad
k\to\infty.
$$
The latter relation follows by combining \eqref{eq.b3}--\eqref{eq.b3a}.
\end{proof}

\subsection{Semiclassical $\Psi$DO on the sphere}
\label{sec.b2}
A semiclassical $\Psi$DO in $L^2(\sph)$ 
can be defined in a variety of ways; 
below we describe a slightly non-standard approach to this definition,
which will simplify our exposition in Section~\ref{sec.c}.

For $\omega,\omega'\in\sph$ such that $\omega+\omega'\not=0$, we set
\begin{equation}
\kappa=\kappa(\omega,\omega')=\frac{\omega+\omega'}{\abs{\omega+\omega'}}
\in\sph.
\label{eq.b11}
\end{equation}
Clearly, $\kappa(\omega,\omega')$ is a smooth function of $(\omega,\omega')\in\sph\times\sph$ 
away from the anti-diagonal 
$$
\AD=\{(\omega,\omega')\in \sph\times\sph \mid \omega+\omega'=0\}.
$$
In order to overcome the (inessential) difficulties related to the singularity of $\kappa$ at
the anti-diagonal, we will assume that our amplitudes vanish in an open neighbourhood
of $\AD$. We will say that a function $b=b(\omega,\omega',\eta)$, $\omega,\omega'\in\sph$,
$\eta\in\Lambda_{\kappa(\omega,\omega')}$, is an \emph{admissible amplitude}, 
if:
\begin{itemize}
\item
$b$ is a $C^\infty$-smooth function of its arguments;
\item
$b(\omega,\omega',\eta)=0$ if
$\abs{\eta}$ is sufficiently large; 
\item
$b(\omega,\omega',\eta)=0$ 
if $(\omega,\omega')$ are in an open neighbourhood of 
$\AD$.  
\end{itemize}

For an admissible amplitude $b$ 
and a semiclassical parameter $k>0$,
we define the operator $\Op_k [b]$ in $L^2(\sph)$  via its integral kernel 
\begin{equation}
\Op_k[b](\omega,\omega')
=
\left(\frac{k}{2\pi}\right)^{d-1}
\int_{\Lambda_{\kappa(\omega,\omega')}} 
e^{-ik\jap{\omega-\omega',\eta}} 
b(\omega,\omega',\eta) \dd\eta,
\label{eq.b13}
\end{equation}
where $\omega,\omega'\in\sph$. 
It is easy to see that for $\omega\not=\omega'$ one has
$$
\Op_k[b](\omega,\omega')=O(k^{-\infty}), \quad k\to\infty.
$$
This shows that the values of the amplitude $b(\omega,\omega',\eta)$ away from an 
open neighbourhood of the diagonal $\omega=\omega'$ 
do not affect the asymptotic properties
of the operator $\Op_k[b]$ as $k\to\infty$.

\begin{proposition}\label{prop.b4}
For any admissible amplitude $b$
and any $k>0$, the operator $\Op_k[b]$ 
is trace class, and for any $\ell\in\N$ one has
$$
\lim_{k\rightarrow\infty}
\left(\frac{k}{2\pi}\right)^{-d+1}
\Tr (\Op_k[b])^{\ell}
=
\int_{\sph}\int_{\Lambda_\omega}b(\omega,\omega,\eta)^{\ell} 
\dd \eta \, \dd \omega.
$$
\end{proposition}
\begin{proof}[Sketch of proof]
The proof follows standard methods of $\Psi$DO theory;
see e.g. \cite[Theorem~9.6]{Dim} for a similar statement in the context of operators
in $\R^n$. Here we only outline the main steps.

First note that for $\ell=1$ the result of Proposition~\ref{prop.b4} is trivial, 
since by a direct evaluation of trace we have the identity
\begin{equation}
\left(\frac{k}{2\pi}\right)^{-d+1}
\Tr (\Op_k[b])
=
\int_{\sph}\int_{\Lambda_\omega}b(\omega,\omega,\eta)
\dd \eta \, \dd \omega.
\label{eq.b14}
\end{equation}
Next, 
using the local coordinates on the sphere and the composition 
formula for symbols of $\Psi$DOs (see e.g. \cite[Proposition~7.7]{Dim}),
we obtain the following statement. 
For any $N>0$ there exists $M>0$ such that $(\Op_k[b])^\ell$ can be 
represented as 
\begin{equation}
(\Op_k[b])^\ell
=
\sum_{m=0}^M k^{-m} \Op_k [b_m] +R_k,
\label{eq.b16}
\end{equation}
where $b_m$ are admissible symbols, $b_0$ is such that
\begin{equation}
b_0(\omega,\omega,\eta)=b(\omega,\omega,\eta)^\ell,
\qquad
\forall\omega\in\sph, \quad \forall\eta\in \Lambda_\omega,
\label{eq.b15}
\end{equation}
and $R_k$ is an integral operator with a smooth kernel which satisfies
$$
\sup_{\omega,\omega'}\abs{R_k(\omega,\omega')}= O(k^{-N}), 
\quad
k\to\infty.
$$
Now taking $N>d-1$ and evaluating the trace in \eqref{eq.b16}, we get
$$
\lim_{k\rightarrow\infty}
\left(\frac{k}{2\pi}\right)^{-d+1}
\Tr (\Op_k[b])^{\ell}
=
\lim_{k\rightarrow\infty}
\left(\frac{k}{2\pi}\right)^{-d+1}
\Tr (\Op_k[b_0])
\\
=
\int_{\sph}\int_{\Lambda_\omega}b_0(\omega,\omega,\eta)
\dd \eta \, \dd \omega.
$$
In view of \eqref{eq.b15}, this proves the required identity. 
\end{proof}

\section{Proof of Theorem \ref{thm.a1}}
\label{sec.c}

\subsection{The Born approximation with $V\in C_0^{\infty}(\R^{d})$}
\label{sec.c1}
\begin{lemma}
\label{prop.c1}
Let $V\in C_0^{\infty}(\R^{d})$. Then for any $\ell\in\N$, one has
\begin{equation}
\label{eq.c1}
\lim_{k\rightarrow\infty}k^{1-d}\Tr(k\Im S_{B}(k^2))^{\ell} = \int_{-\infty}^{\infty}t^{\ell}d\mu(t).
\end{equation}
\end{lemma}
\begin{proof}
1. 
For ease of notation we write $Q(k)=k\Im S_{B}(k^2)$.
By  \eqref{eq.a9} and \eqref{eq.a12}, $Q(k)$ is the integral operator 
in $L^2(\sph)$ with the integral kernel 
$$
Q(k)(\omega,\omega')
=
-\frac{1}{2} \left(\frac{k}{2\pi}\right)^{d-1}
\int_{\R^{d}}e^{-ik\jap{\omega-\omega',x}}V(x)\dd x, 
\quad
\omega,\omega'\in\sph.
$$
For fixed $\omega$, $\omega'$, $\omega+\omega'\not=0$, let $\kappa=\kappa(\omega,\omega')$ be
as in \eqref{eq.b11}.
Write any $x\in\R^{d}$ as 
$x=\kappa t+\eta$ with $t\in\R$, $\eta\in\Lambda_\kappa$. 
Note that by the orthogonality relation $(\omega-\omega')\perp \kappa$, 
one has
$$
\jap{\omega-\omega',x}
=
\jap{\omega-\omega',\eta}.
$$
Thus, the integral kernel of $Q(k)$ can be rewritten as
\begin{multline}
Q(k)(\omega,\omega')
=
-\frac{1}{2} \left(\frac{k}{2\pi}\right)^{d-1}
\int_{\Lambda_{\kappa(\omega,\omega')}}\int_{-\infty}^\infty
e^{-ik\jap{\omega-\omega',\eta}}
V(\kappa(\omega,\omega')t+\eta) \dd t\, \dd \eta
\\
=
\left(\frac{k}{2\pi}\right)^{d-1}
\int_{\Lambda_{\kappa(\omega,\omega')}}
e^{-ik\jap{\omega-\omega',\eta}}
X(\kappa(\omega,\omega'),\eta)\dd \eta, 
\label{c3}
\end{multline}
where $X$ is given by \eqref{eq.a5}. 
From here we directly obtain the required 
identity \eqref{eq.c1} in the case $\ell=1$ by 
integrating the kernel of $Q(k)$ over the diagonal. 
Now it remains to prove \eqref{eq.c1} for $\ell\geq2$.

2.
Let $\chi_0\in C_0^\infty(\sph\times\sph)$ 
be such that $\chi_0(\omega,\omega')=1$ in an open neighbourhood
of the diagonal $\omega=\omega'$ and $\chi_0(\omega,\omega')=0$
in an open neighbourhood of the anti-diagonal $\omega+\omega'=0$. 
Denote $\chi_1=1-\chi_0$, and let
$$
Q(k)=Q_0(k)+Q_1(k),
$$
where $Q_j(k)$ is the operator with the integral kernel 
$\chi_j(\omega,\omega')Q(k)(\omega,\omega')$. 
By the fast decay of the Fourier transform of $V$ and by the fact that
$\abs{\omega-\omega'}$ is separated away from zero on the support
of $\chi_1$, we see that 
$$
\sup_{\omega,\omega'}\abs{Q_1(k)(\omega,\omega')}=O(k^{-\infty}), 
\quad 
k\to\infty.
$$
Thus, it suffices to prove that
$$
\lim_{k\to\infty}k^{1-d}\Tr (Q_0(k))^{\ell}
=
\int_{-\infty}^\infty t^{\ell} d\mu(t).
$$
From \eqref{c3} 
it follows that $Q_0(k)$ can be represented as a semiclassical 
$\Psi$DO on the sphere of the type \eqref{eq.b13}: 
$$Q_0(k)=\Op_k[b],
\quad \text{where }\quad
b(\omega,\omega',\eta)=\chi_0(\omega,\omega') X(\kappa(\omega,\omega'),\eta),
$$ 
$b$ is an admissible amplitude in the sense discussed in Section~\ref{sec.b2},
and  $X$ is given by \eqref{eq.a5}.
Applying Proposition~\ref{prop.b4}, we get
\begin{equation}
\lim_{k\to\infty}k^{1-d}\Tr (Q_0(k))^{\ell}
=
(2\pi)^{1-d}\int_{\sph}\int_{\Lambda_\omega} X(\omega,\eta)^{\ell}  \dd \eta\,\dd \omega
\\
=
\int_{-\infty}^\infty t^{\ell} d\mu(t),
\label{c2a}
\end{equation}
as required. 
\end{proof}

\subsection{The Born approximation with general $V$}
\label{sec.c2}

\begin{lemma}\label{prop.c2}
Let $V\in X_\rho$ with $\rho>1$. 
Then for any integer $\ell\geq 1$ satisfying $\ell > \frac{d-1}{\p-1}$
one has
$$
\lim_{k\rightarrow\infty}k^{1-d}\Tr(\Im kS_{B}(k^2))^{\ell} = \int_{-\infty}^{\infty}t^{\ell}\dd\mu(t).
$$
\end{lemma}
\begin{proof}
Let $X_{\rho}^{0}$ be the closure of $C_0^{\infty}(\R^{d})$ in $X_{\rho}$. 
For any $\ell>\frac{d-1}{\rho-1}$, denote
\begin{align*}
g_{\ell}(V)& = \int_{-\infty}^{\infty}t^{\ell}\dd\mu(t), \\
g_{\ell}^{+}(V)& = \limsup_{k\rightarrow\infty}k^{1-d}\Tr(k\Im S_{B}(k^2))^{\ell}, \\
g_{\ell}^{-}(V)& = \liminf_{k\rightarrow\infty}k^{1-d}\Tr(k\Im S_{B}(k^2))^{\ell}.
\end{align*}
By Lemma~\ref{prop.c1}, for all $V\in C_0^{\infty}(\R^{d})$ we have
\begin{equation}
\label{eq.c18}
g_{\ell}(V) = g_{\ell}^{+}(V)=g_{\ell}^{-}(V).
\end{equation}
Recall that $S_B(k^2)$ depends linearly on $V$. 
Using  this fact and the estimates \eqref{eq.b3} and \eqref{eq.b10}, 
it is easy to check that $g_{\ell}^{\pm}(V)$ are continuous functionals on $X_{\p}$. 
Next, using the last equality in \eqref{c2a} and 
 the estimate \eqref{eq.a6}, it is easy to see that $g_{\ell}(V)$ is a continuous functional on $X_{\p}$. 
Thus, by an approximation argument,  \eqref{eq.c18} holds for any $V\in X_{\p}^{0}$. 
Finally, for any $V\in X_{\rho}$ and a given $\ell>\frac{d-1}{\rho-1}$, choose $\rho_{1}$ 
such that $1<\rho_{1}<\rho$ with $\ell>\frac{d-1}{\rho_{1}-1}$. 
Then $X_{\rho}\subset X_{\rho_{1}}^{0}$ and the previous argument proves
\eqref{eq.c18} for all $V\in X_{\rho_{1}}^{0}$ which suffices.
\end{proof}
\subsection{From the Born approximation to the full scattering matrix}
\label{sec.c3}
\begin{lemma}
\label{prop.c3}
Let $V\in X_\rho$ with $\rho>1$. 
Then for any integer $\ell\geq 1$ satisfying $\ell+2>\frac{d-1}{\p-1}$,
$$
\lim_{k\rightarrow\infty}k^{1-d}\int_{-\infty}^{\infty}t^{\ell}\dd\mu_{k}(t) = \int_{-\infty}^{\infty}t^{\ell}\dd\mu(t).
$$
\end{lemma}
\begin{proof}
By Lemmas~\ref{prop.b2} and \ref{prop.c2}, it suffices to prove
\begin{equation}
\label{eq.c20}
\lim_{k\rightarrow\infty}k^{1-d}\left\vert
\int_{-\infty}^{\infty}t^{\ell}\dd\mu_{k}(t) - \Tr(\Im kS(k^2))^{\ell}
\right\vert = 0.
\end{equation}
Recalling the definition \eqref{eq.a3} of the measure $\mu_k$, 
one sees that \eqref{eq.c20} is equivalent to 
\begin{equation}
\label{eq.c21}
\lim_{k\rightarrow\infty}k^{1-d+\ell}\left\vert\sum_{n=1}^{\infty}[(\theta_{n}(k^2))^{\ell} - (\sin\theta_{n}(k^2))^{\ell}]\right\vert = 0.
\end{equation}
By \eqref{eq.a2} we have  $0<\vert\theta_{n}(k^2)\vert <\pi/4$ for all sufficiently large $k$ and all $n$. 
From the elementary estimates $\vert\theta_{n}\vert\leq 2\vert\sin\theta_{n}\vert$ and $\vert\theta_{n} - \sin\theta_{n}\vert\leq C\vert\sin\theta_{n}\vert^{3}$ which hold for $\vert\theta_{n}\vert <\pi/4$, it follows that for all sufficiently large $k$
\begin{equation*}
\begin{split}
k^{1-d+\ell}\sum_{n=1}^{\infty}&\vert(\theta_{n})^{\ell} - (\sin\theta_{n})^{\ell}\vert 
\leq 
k^{1-d+\ell}\sum_{n=1}^{\infty}
\left(
\vert\theta_{n}-\sin\theta_{n}\vert
\sum_{j=0}^{\ell-1}\vert\theta_{n}\vert^{j}\vert\sin\theta_{n}\vert^{\ell-1-j}
\right)
\\
&\leq 
k^{1-d+\ell}C(\ell)\sum_{n=1}^{\infty}\vert\sin\theta_{n}\vert^{\ell+2}
=
k^{1-d-2}C(\ell)\Vert\Im kS(k^{2})\Vert_{\ell+2}^{\ell+2}.
\end{split}
\end{equation*}
Now \eqref{eq.c21} follows by applying the estimate \eqref{eq.b3a} for 
$\Vert\Im kS(k^{2})\Vert_{\ell+2}$ to the result just obtained.
\end{proof}
\begin{proof}[Proof of Theorem \ref{thm.a1}]
By the estimate \eqref{eq.a2}, the supports of $\mu_k$ are bounded uniformly 
in $k\geq1$. On the other hand, by the boundedness of $V$, the support of $\mu$
is also bounded. 
Thus, we may choose $T>0$ such that 
$$
\supp\mu\subset[-T,T]
\quad\text{ and }\quad
\supp\mu_k\subset[-T,T] 
\quad\text{ for all }k\geq1.
$$

Next, fix $\psi\in C_0^{\infty}(\R\setminus\{ 0\})$, 
and let $\ell_{0}$ be an even natural number satisfying 
$\ell_{0}>\frac{d-1}{\p-1}$. 
Since $\psi(t)$ vanishes near $t=0$ by assumption, 
the function $\psi(t)/t^{\ell_{0}}$ is smooth. 
By the Weierstrass approximation theorem, for any $\e >0$ 
there exists a polynomial $\psi_0(t)$ such that
$$
\abs{\psi(t)t^{-\ell_0}-\psi_0(t)}\leq \e, 
\quad \forall t\in[-T,T].
$$
Denoting $\psi_\pm(t)=(\psi_0(t)\pm\e)t^{\ell_0}$,
we obtain
\begin{align}
&\psi_-(t)\leq \psi(t)\leq \psi_+(t), \quad \forall t\in[-T,T],
\label{eq.c22}
\\
&\psi_+(t)-\psi_-(t)=2\e t^{\ell_0}.
\label{eq.c23}
\end{align}
By \eqref{eq.c22}, we get
\begin{equation}
\int_{-\infty}^{\infty}\psi_{-}(t)\dd\mu_{k}(t)
\leq
\int_{-\infty}^{\infty}\psi(t)\dd\mu_{k}(t)
\leq
\int_{-\infty}^{\infty}\psi_{+}(t)\dd\mu_{k}(t).
\label{eq.c24}
\end{equation}
By construction, $\psi_\pm(t)$ are polynomials which involve
powers $t^m$ with $m\geq \ell_0$. Thus, we can apply 
Lemma~\ref{prop.c3} to \eqref{eq.c24}, 
which yields 
\begin{align*}
\limsup_{k\rightarrow\infty}k^{1-d}\int_{-\infty}^{\infty}\psi(t)\dd\mu_{k}(t)
&\leq 
\int_{-\infty}^{\infty}\psi_{+}(t)\dd\mu(t),
\\
\liminf_{k\rightarrow\infty}k^{1-d}\int_{-\infty}^{\infty}\psi(t)\dd\mu_{k}(t)
&\geq 
\int_{-\infty}^{\infty}\psi_{-}(t)\dd\mu(t).
\end{align*}
On the other hand, by \eqref{eq.c22}, \eqref{eq.c23}, 
$$
\int_{-\infty}^{\infty}\psi_-(t)\dd\mu(t)
\leq 
\int_{-\infty}^{\infty}\psi(t)\dd\mu(t)
\leq
\int_{-\infty}^{\infty}\psi_+(t)\dd\mu(t)
$$
and 
$$
\int_{-\infty}^{\infty}\psi_+(t)\dd\mu(t)
-
\int_{-\infty}^{\infty}\psi_-(t)\dd\mu(t)
=
2\e 
\int_{-\infty}^{\infty}t^{\ell_0}\dd\mu(t).
$$
By \eqref{eq.a6}, the integral in the right hand side of the last estimate is finite; 
denote this integral by $C$.
Combining the above estimates, we obtain 
\begin{align*}
\limsup_{k\rightarrow\infty}k^{1-d}\int_{-\infty}^{\infty}\psi(t)\dd\mu_{k}(t)
&\leq 
\int_{-\infty}^{\infty}\psi(t)\dd\mu(t)
+
2\e C,
\\
\liminf_{k\rightarrow\infty}k^{1-d}\int_{-\infty}^{\infty}\psi(t)\dd\mu_{k}(t)
&\geq 
\int_{-\infty}^{\infty}\psi(t)\dd\mu(t)
-
2\e C.
\end{align*}
Since $\e>0$ can be taken arbitrary small, we obtain the required
statement.
\end{proof}

\goodbreak

\end{document}